\documentclass[11pt]{article}
\usepackage[mathvars]{brsheader}

\title{New optima for the deletion shadow}
\author{Benedict Randall Shaw}
\date{May 2025}

\begin{document}

\maketitle

\begin{abstract}
    For a family \(\mathcal{F}\) of words of length \(n\) drawn from an alphabet \(A=[r]=\{1,\dots,r\}\), Danh and Daykin defined the deletion shadow \(\Delta \mathcal{F}\) as the family containing all words that can be made by deleting one letter of a word of \(\mathcal{F}\). They asked, given the size of such a family, how small its deletion shadow can be, and answered this with a Kruskal-Katona type result when the alphabet has size \(2\). However, Leck showed that no ordering can give such a result for larger alphabets. The minimal shadow has been known for families of size \(s^n\), where the optimal family has form \([s]^n\). We give the minimal shadow for many intermediate sizes between these levels, showing that families of the form `all words in \([s]^n\) in which the symbol \(s\) appears at most \(k\) times' are optimal. This proves a conjecture of Bollobás and Leader. Our proof uses some fractional techniques that may be of independent interest.
\end{abstract}

\section{Introduction}

We consider words composed of exactly \(n\) symbols from a finite alphabet \(A\), which we will typically take to be \([r]=\{1,\dots,r\}\). We say a family \(\mathcal{F}\subset A^n\) has \emph{deletion shadow in direction} \(i\)
\[\delta_i\mathcal{F}=\left\{(x_1,\dots,x_{i-1},x_{i+1},\dots,x_n):(x_1,\dots,x_n)\in\mathcal{F}\right\},\]
and \emph{deletion shadow}
\[\Delta\mathcal{F}=\bigcup_{i=1}^n\delta_i\mathcal{F}.\]

For example, the family \(\mathcal{F}=\{112,123\}\) has deletion shadow \(\Delta\mathcal{F}=\{11,12,13,23\}\).

This notion was introduced by Danh and Daykin \cite{danh_ordering_1997}. Motivated by the Kruskal-Katona theorem, they asked the following question:

\begin{question}
    Given \(|\mathcal{F}|\), how small can \(|\Delta\mathcal{F}|\) be?
\end{question}

It is folklore that families of the form \([s]^n\) have minimal shadow given their size (a proof of which we outline later). However, this only gives the minimal deletion shadow for families of size \(1^n,2^n,\dots,r^n\), which is a very sparse set of values for \(r\) fixed and \(n\) large.

We define the \emph{simplicial order} for \(\{0,1\}^n\) by taking \(x<y\) if \(x\) contains the symbol \(1\) fewer times than \(y\), or if \(x\) and \(y\) each contain the symbol \(1\) the same number of times and at the first \(i\) such that \(x_i\neq y_i\), we have \(x_i=1\) and \(y_i=0\). Regarding words in \(\{0,1\}^n\) as indicator functions on \([n]\), this definition agrees with the notion of simplicial order on subsets of \([n]\). For an alphabet of two symbols, Danh and Daykin \cite{danh_ordering_1997}, \cite{danh_sets_1997} proved an equivalent of the Kruskal-Katona theorem for the deletion shadow.

\begin{theorem}[\cite{danh_ordering_1997}]
    For \(A=\{0,1\}\), across all families \(\mathcal{F}\subset A^n\) of a given size, the initial segment of the simplicial order minimises the deletion shadow.
\end{theorem}

Danh and Daykin \cite{danh_structure_1996} defined a more complex order for general \(r\), the V-order, which they studied hoping a similar result would hold---that every initial segment of the V-order would have minimal deletion shadow for its size. However, Leck \cite{leck_nonexistence_2004} proved that for \(r\geq 3\), such a result is false not only for the V-order, but in fact for any order. This forbids an exact ordering result of this form---in general, no order can have all initial segments minimise their deletion shadow given their size.

However, it leaves open the possibility of an ordering \textit{many} of whose initial segments have minimal shadow given their size. Our main result will be the following, proving a conjecture of Bollobás and Leader (see \cite{raty_coordinate_2019}), which gives the optimal family for a much denser set of values than previously known:

\begin{theorem}
    For any \(s\leq r\), \(k\leq n\), the family \(\mathcal{F}\) comprising all words of \([s]^n\) that contain the symbol \(s\) at most \(k\) times has minimal deletion shadow of all families in \([r]^n\) of the same size.\label{thm:discreteresult}
\end{theorem}

We will show this by considering a version of the problem in the setting of fractional families, a more general notion in which each word has a weight between \(0\) and \(1\), where weights \(0\) and \(1\) correspond respectively to non-membership and membership in discrete families. In this relaxed setting, we will be able to find the family with minimal deletion shadow of each weight. This guarantees that our argument does not come up against Leck's result, in that while we will be able to answer the question exactly in our generalisation, these fractional families will only correspond to discrete families at sizes matching the families in Theorem \ref{thm:discreteresult}.

We will have to be careful in choosing how to generalise the deletion shadow to fractional families in a way that interacts nicely with the minima we prove. Indeed, there are several natural definitions of a deletion shadow on fractional families. Perhaps counterintuitively, we choose a definition which in general is minimised not by the families we care about, but rather by the \textit{uniform fractional families}, whose indicator functions are constant. 

However, it is known that in the discrete setting, we may assume that our families are down-compressed. We therefore restrict to fractional families obeying a certain weaker compression condition. Again there are various ways to apply such a condition, and so the choice of this condition is key to the proof. A relatively standard argument then determines the minimal shadow of a fractional family of a given weight, subject to our restrictions. As an immediate consequence, we recover Theorem \ref{thm:discreteresult} in the discrete setting.

For other results on the deletion shadow, see Danh and Daykin \cite{danh_structure_1996,danh_ordering_1997,danh_sets_1997}. Other notions of the deletion shadow have also been studied. For example, Räty \cite{raty_coordinate_2019} proved an exact ordering result for the related shadow obtained by deleting only the character \(1\).

We finish this section by outlining the folklore proof that \([s]^n\) has minimal shadow given its size. We can view \(\delta_i\mathcal{F}\) as a projection of \(\mathcal{F}\) onto the hyperplane defined by \(x_i=0\). Then the discrete Loomis-Whitney inequality \cite{loomis_inequality_1949}, applied to \(\mathcal{F}\), states that
\[|\mathcal{F}|\geq\prod_{i=1}^n|\delta_i\mathcal{F}|^{\frac1{n-1}}.\]
As \(\delta_i\mathcal{F}\subset\Delta\mathcal{F}\), this gives the bound \(|\Delta\mathcal{F}|\geq |\mathcal{F}|^{\frac{n-1}{n}}\). This is attained by \([s]^n\). In particular, this implies that for \(s\leq r\), the family \(\mathcal{F}=[s]^n\) has minimal deletion shadow of all families of that size. 

In Section \ref{sec:overview}, we give a more detailed overview of the proof. In Section \ref{sec:down}, we present a version of Danh and Daykin's argument that we may assume discrete families with minimal shadow are down-compressed. In Section \ref{sec:frac}, we set up the the problem in the context of fractional families, and prove a more general result there which implies Theorem \ref{thm:discreteresult}. Finally, in Section \ref{sec:further}, we conjecture some other discrete families of minimal deletion shadow given their size.
\section{Overview}\label{sec:overview}

We first note that by an argument of Danh and Daykin \cite{danh_ordering_1997}, we may assume that the family \(\mathcal{F}\) is down-compressed---that is to say, if \(x_1\dots x_n\) is a word in \(\mathcal{F}\), then any word \(x'_1\dots x'_n\) with \(x'_i\leq x_i\) for each \(i\) is also in \(\mathcal{F}\). We present a version of their argument for completeness, although it will not be necessary to follow the rest of the proof.

Leck \cite{leck_nonexistence_2004} proved that there is no exact ordering of words whose initial segments each have minimal deletion shadow for their size. Thus any proof that would give us a result of this form must fail. We therefore work instead in fractional families, introduced by Bollobás and Leader in \cite{bollobas_isoperimetric_1991,bollobas_maximal_1993}. In this setting, words cannot just be members and non-members of the family, but may be members of the family to any extent in \([0,1]\). These allow us to move uniformly between the discrete families that we show to be optimal in Theorem \ref{thm:discreteresult} without saying anything at all about the optimal discrete families between those sizes.

We extend the notion of a deletion shadow accordingly. There are several possible ways to define the deletion shadow on fractional families, and we must choose a definition that interacts well with the problem. In fact, the version we use is minimised by the uniform fractional family which contains every set in \(A^n\) to the same extent. We would like to show that certain discrete families are optimal even in the wider setting of fractional families. Unfortunately, the uniform fractional family is not a discrete family.

However, we know we can assume an optimal discrete family is down-compressed, and the uniform fractional family is not down-compressed. Thus we instead minimise the deletion shadow across fractional families satisfying a certain weaker notion of down-compression, which we call \(s\)-\emph{compression}. Subject to this condition, we will be able to compute the exact minimal deletion shadow, which will imply Theorem \ref{thm:discreteresult}.

The proof is an induction on \(n\), essentially by a codimension-\(1\) compression in dimension \(n\). We divide the fractional family \(f\) into its layers
\[f_y(x_1,\dots,x_{n-1})=f(x_1,\dots,x_{n-1},y).\]
We first treat the case where \(f\) is zero outside \([s]^n\), where \(s\) is minimal such that \(f\) has weight at most \(s^n\). We define a notion of a fractional Hamming ball, and consider a fractional family \(g\) with the same weight as \(f\) so that each layer \(g_x\) is a fractional Hamming ball, \(g_s\) has the same weight as \(f_s\), and \(g_1,\dots,g_{s-1}\) each have the same weight. Our compression condition will tell us that \(g_s\) has weight at most that of \(g_1\). We will then show that the deletion shadow of \(f\) has weight at least as large as that of the fractional Hamming ball \(h\) with the same weight as \(f\).

We then distinguish three regimes depending on the weight of \(g_s\). If it has weight at most as large as that of \(h_s\), then we find that \(\Delta g\) has the same weight as \(g\)'s heaviest layer. We observe that by deletion in the last coordinate, \(f\) has shadow of weight at least as large as \(f\)'s largest layer, and thus has shadow of weight at least as large as that of \(g\). We then note that \(g\)'s shadow has weight at least as large as that of \(h\).

If \(g_s\) is heavier than \(h_s\), but \(g\) still meets our compression condition, we notice that the shadow of \(g\) is exactly the same as that obtained from deletion in non-final coordinates. We note that \(f\)'s shadow is at least as large as that obtained from \(f\) by deletion in non-final coordinates. Our induction on \(n\) show that this bound is minimised by \(g\). Since \(g\) attains this bound, it is optimal given the weight of \(g_s\). We then calculate that \(g\)'s shadow has weight at least as large as that of \(h\).

Finally, we consider the case where \(g_s\) is heavy enough that \(g\) no longer meets our compression condition. Here the bound of the previous case still applies to this one, and moreover its value does not depend on the exact distribution of weight between layers. This allows us to move our weight between until we are in the previous case, where the compression condition is satisfied, and simply apply the result for that case directly.

We finish by checking the case where \(f\) is not a subfamily of \([s]^n\). Here we write down an expression for the weight of \(\Delta f\) in terms of the weights of \(f_x\) with \(x> s\), and find it is a simple optimisation problem.

\section{Down-compression}\label{sec:down}

We say \(\mathcal{F}\subset A^n\) is \emph{down-compressed} if for any \(i\in[n]\),
\((x_1,\dots,x_n)\in\mathcal{F}\) and \(y\in A\) such that \(y<x_i\), we have \((x_1,\dots,x_{i-1},y,x_{i+1},\dots,x_n)\in \mathcal{F}\). We will need the following lemma, due to Danh and Daykin \cite{danh_ordering_1997}:

\begin{lemma}
    For any \(\mathcal{F}\subset A^n\), there exists a down-compressed \(\mathcal{F}'\) with \(|\mathcal{F}'|=|\mathcal{F}|\) such that \(|\Delta\mathcal{F}'|\leq |\Delta\mathcal{F}|\).\label{lem:downcompression}
\end{lemma}

For completeness, we present a version of their proof for the interested reader. The details will not be needed for the sections following. We make some rather heavy syntactic conventions, for which we apologise---these are not chosen to confuse the reader, but to allow the proof to be written down in a precise way.

\begin{proof}[Proof of Lemma \ref{lem:downcompression}]
For \(a,b\in A\) with \(a<b\), \(i\in [n]\), we define the compression \(C^{a,b}_i\mathcal{F}\) from \(\mathcal{F}\) by replacing any \((x_1,\dots,x_{i-1},b,x_{i+1},\dots,x_n)\) with \((x_1,\dots,x_{i-1},a,x_{i+1},\dots,x_n)\) unless it is already in the family \(\mathcal{F}\). Then clearly \(\left|C^{a,b}_i\mathcal{F}\right|=|\mathcal{F}|\).

Write \(C^{a,b}\mathcal{F}=C^{a,b}_1\cdots C_n^{a,b}\mathcal{F}\). Then it will suffice to show that \(\left|\Delta C^{a,b}\mathcal{F}\right|\leq |\Delta\mathcal{F}|\), as repeatedly applying different compressions of the form \(C^{a,b}\) until none of them change \(\mathcal{F}\) further eventually gives a down-compressed family. (Note that repeatedly applying these compressions must eventually result in a family that is fixed by all of them, as any such compression that changes \(\mathcal{F}\) must decrease the natural number \(\sum_{i\in[n],w\in\mathcal{F}}w_i\). Such a family is down-compressed, although it may depend on the order in which compressions are chosen.)

We now define, for each \(x\in A\), a quantifier \(Q_x\) which will allow us to write a concise condition for a given word to be a member of \(C^{a,b}\mathcal{F}\). For a sentence \(\phi=\phi(x_1,\dots,x_n,y_1,\dots,y_n)\), we will write \(Q_ay.\phi(y,x_1,\dots)\) if for some choice of \(y\in\{a,b\}\) the sentence \(\phi(y,x_1,\dots)\) holds---that is to say, \(Q_ay\) corresponds to the quantifier \(\exists y\in \{a,b\}\). We will write \(Q_by.\phi(y,x_1,\dots)\) if \(\phi(y,x_1,\dots)\) holds for both choices of \(y\in\{a,b\}\)---that is to say, \(Q_by\) corresponds to the quantifier \(\forall y\in \{a,b\}\). Finally, when \(x\in A\setminus\{a,b\}\), we write \(Q_xy.\phi(y,x_1,\dots)\) if \(\phi(y,x_1,\dots)\) holds when \(y=x\). Note that, all of these quantifiers imply \(\exists y . \phi(y,x_1,\dots)\).

For example, suppose that \(a=3\) and \(b=4\). Then the statement
\[Q_1w.Q_3x.Q_2y.Q_4z.\phi\left(w,x,y,z\right)\]
is equivalent to the statement
\[\exists x\in \{3,4\}.\forall z \in \{3,4\}.\phi(1,x,2,z).\]

Now notice that \((x_1,\dots,x_n)\in C^{a,b}\mathcal{F}\) if and only if \(Q_{x_1}y_1. (y_1,x_2,\dots,x_n)\in C_2^{a,b}\cdots C_n^{a,b}\mathcal{F}\). Inductively, we see that \((x_1,\dots,x_n)\in C^{a,b}\mathcal{F}\) whenever \(Q_{x_1}y_1\dots Q_{x_n}y_n. (y_1,\dots,y_n)\in \mathcal{F}\). However, for any formula \(\phi\), whenever \(Q_{x}y.\phi(y,x_1,\dots)\) we also have \(\exists y.\phi(y,x_1,\dots)\).

Now any word in \(\Delta C^{a,b}\mathcal{F}\) has the form \((x_1,\dots,x_{i-1},x_{i+1},\dots,x_n)\) for some \((x_1,\dots,x_n) \in C^{a,b} \mathcal{F}\). Then \(Q_{x_1}y_1\dots Q_{x_n}y_n. (y_1,\dots,y_n)\in \mathcal{F}\). Hence
\[Q_{x_1}y_1\dots Q_{x_{i-1}}y_{i-1}.\exists y_i.Q_{x_{i+1}}y_{i+1}\dots Q_{x_n}y_n. (y_1,\dots,y_n)\in \mathcal{F}.\]
And so
\[Q_{x_1}y_1\dots Q_{x_{i-1}}y_{i-1}.Q_{x_{i+1}}y_{i+1}\dots Q_{x_n}y_n. (y_1,\dots,y_{i-1},y_{i+1},\dots,y_n)\in \Delta\mathcal{F}.\]
But now we have \((x_1,\dots,x_{i-1},x_{i+1},\dots,x_n)\in C^{a,b}\Delta\mathcal{F}\), as desired.

Thus \(\Delta C^{a,b}\mathcal{F}\subset C^{a,b}\Delta\mathcal{F}\), so we must have \(\left|\Delta C^{a,b}\mathcal{F}\right|\leq |\Delta\mathcal{F}|\), giving the desired result.

\end{proof}

\section{The fractional setting}\label{sec:frac}

We can identify a family \(\mathcal{F}\subset A^n\) with its \emph{indicator function}, the function from \(A^n\) to the set \(\{0,1\}\) defined by \(f(x)=1\) for each \(x\in\mathcal{F}\), \(f(x)=0\) otherwise. We will work in fractional families, in which an element may be `partially present'. We define a \emph{fractional family} as any function from \(A^n\) to the unit interval \([0,1]\). This will play the role that indicator functions played in discrete (i.e.\ non-fractional) families. We say a family \(f\) has \emph{weight}
\[|f|=\sum_{w\in A^n}f(w).\]
Note that when \(f\) is the indicator function of a discrete family \(\mathcal{F}\), we have \(|f|=|\mathcal{F}|\). We will say a family \(f\) is a \emph{subfamily} of \(g\) if for every \(w \in A^n\), we have \(f(w)\leq g(w)\). Given a fractional family \(f\), we define the family \(\Delta f\) by
\[\left(\Delta f\right)(x_1,\dots,x_{n-1})=\max_{y\in A,i\in[n]}f(x_1,\dots,x_{i-1},y,x_i,\dots,x_{n-1}).\]
Again, if \(f\) is the indicator function of a discrete \(\mathcal{F}\subset A^n\), then \(\Delta f\) is the indicator function of \(\Delta \mathcal{F}\). For example, where we took a maximum, we could instead have taken a sum capped at \(1\), obtaining a related shadow \(\Delta_\Sigma f\). Another possible definition would be
\[\left(\Delta_{\Sigma\max} f\right)(x_1,\dots,x_{n-1})=\min\left\{\max_{i\in[n]}\sum_{y\in A}f(x_1,\dots,x_{i-1},y,x_i,\dots,x_{n-1}),1\right\}.\]
Unfortunately, our chosen definition of the deletion shadow is minimised, for a given weight of \(f\), by the \emph{uniform fractional family} which has \(f(w)\) constant, which gives no information about discrete families. This is not true of \(\Delta_\Sigma\) and \(\Delta_{\Sigma \max}\), making them attractive candidates to work with. However, we will see that \(\Delta\) interacts better with the families we show are optimal.

In discrete families, the \emph{Hamming ball} \(\mathcal{B}^{(n,s)}(k)\) is the family of words in \([s]^n\) containing the symbol \(s\) at most \(k\) times. We extend this notion to \textit{fractional balls}, whose indicator function is defined to be \(1\) on words containing the symbol \(s\) at most \(k\) times, some fixed \(\alpha\in [0,1]\) on words containing the symbol \(s\) at most \(k+1\) times, and \(0\) otherwise. We will call this \(b^{(n,s)}_{k,\alpha}\), or sometimes just \(b_{k,\alpha}\) where the meaning is clear. Note that \(b_{k-1,1}\) and \(b_{k,0}\) are the same family, which correspond to the discrete family \(\mathcal{B}^{(n,s)}(k)\).

These balls interact extremely well with the definition of the deletion shadow we have chosen, in that \(\Delta b^{(n,s)}_{k,\alpha}=b^{(n-1,s)}_{k,\alpha}\). This is not true of our other possible deletion shadows, where we find \(\Delta_\Sigma b_{k,\alpha}=b_{k,\min\{n(s-1)\alpha,1\}}\), and \(\Delta_{\Sigma\max} b_{k,\alpha}=b_{k,\min\{(s-1)\alpha,1\}}\). (Note that these shadows also decrease the word length \(n\) by \(1\), even though we have not marked this explicitly.) In particular, when \(\alpha=\frac{1}{s-1}\), any family \(f\) which has \(b_{k,\alpha}(w)\leq f(w) \leq b_{k,1}(w)\) for each word \(w\) would have the same shadow \(b_{k,1}\) in either \(\Delta_\Sigma\) or \(\Delta_{\Sigma \max}\), so if the fractional ball is optimal in these shadows, it is far from unique.

In order to work in \(\Delta\), we will need to restrict the families we optimise over to avoid the uniform fractional family. We know from Lemma \ref{lem:downcompression} that in discrete families, it suffices to consider down-compressed families. A natural extension of down-compression to fractional families is \emph{solid down-compression}: we say a family \(f\) on \(A^n\) is \emph{solidly down-compressed} if for any \(x_1,\dots,x_n,y\in A\) and \(i\in [n]\) such that \(y<x_i\), we have \[f(x_1,\dots,x_n)>0 \quad \implies \quad f(x_1,\dots,x_{i-1},y,x_{i+1},\dots,x_n)=1.\]
Indeed, restricted to solidly down-compressed families, \(\Delta\) and \(\Delta_{\Sigma\max}\) agree. 

However, the fractional ball is not generally down-compressed. Since we aim to show the fractional ball is optimal, we will need to work in a restricted version of compression which the fractional ball satisfies. For \(\ell<r\), we say a family \(f\) on \(A^n\) is \(\ell\)-\emph{compressed} if, for any \(x_1,\dots,x_n,y\in A\) and with \(i\in[n]\) such that \(y<x_i\) and \(\ell\leq x_i\), we have
\[f(x_1,\dots,x_n)>0 \quad \implies \quad f(x_1,\dots,x_{i-1},y,x_{i+1},\dots,x_n)=1.\] We prove the following:

\begin{theorem}
    Let \(f\) be an \(s\)-compressed fractional family on \(A^n\) with \(|f|\leq s^n\), and let \(h\) be the fractional Hamming ball on \(A^n\) with the same weight as \(f\). Then \(|\Delta f|\geq|\Delta h|\).\label{thm:fractionalresult}
\end{theorem}

Notice that any down-compressed discrete family is \(\ell\)-compressed for every \(\ell\). Thus by Lemma \ref{lem:downcompression}, this result implies Theorem \ref{thm:discreteresult}. The compression condition is essential in this theorem, as otherwise the minimal shadow is obtained by a totally uniform fractional family, i.e.\ one where \(f\) is constant. 
\begin{proof}[Proof of Theorem \ref{thm:discreteresult} from Theorem \ref{thm:fractionalresult}]
Suppose that the discrete family \(\mathcal{F}\) has minimal deletion shadow across all families with the same size as \(\mathcal{B}^{(n,s)}(k)\). By Lemma \ref{lem:downcompression}, we may assume that \(\mathcal{F}\) is down-compressed. Consider the indicator function \(f\) of \(\mathcal{F}\), and notice that as \(\mathcal{F}\) is down-compressed, certainly \(f\) is \(s\)-compressed, as down-compression is a stronger condition. We also have \(|f|\leq s^n\).

But now Theorem \ref{thm:fractionalresult} tells us that \(|\Delta f|\) is at least as large as \(|\Delta h|\), where \(h\) is the fractional Hamming ball on \(A^n\) with the same weight as \(f\). But then \(h\) is the indicator function of \(\mathcal{B}^{(n,s)}(k)\). Hence
\[\left|\Delta \mathcal{F}\right|=|\Delta f|\geq |\Delta h|=\left|\Delta \mathcal{B}^{(n,s)}(k)\right|.\]
Hence \(\mathcal{B}^{(n,s)}(k)\) has minimal deletion shadow of all discrete families in \([r]^n\) of the same size, as desired.
\end{proof}

\begin{proof}[Proof of Theorem \ref{thm:fractionalresult}]
    For \(x\in A\), define the family \(f_x:A^{n-1}\to[0,1]\) by
    \[f_x(x_1,\dots,x_{n-1})=f(x_1,\dots,x_{n-1},x).\]
    Note that by considering deletion in the last coordinate, for any \(x\) we find that \(f_x\) is a subfamily of \(\Delta f\). In particular, if some \(|f_x|>s^{n-1}\), then
    \[|\Delta f|\geq |f_x|> s^{n-1}\geq|\Delta h|,\]
    so we may assume that \(|f_x|\leq s^{n-1}\) for each \(x\).
    
    We first treat the case where \(f_x=0\) for all \(x>s\). We will then be able to use the result in this case to solve the remaining case as a simple optimisation problem.
    
    We proceed by induction on \(n\). Notice that for \(n=1\), we have \(|\Delta f|=\max\{f(1)\dots,f(r)\}\), and the condition of \(s\)-compression says that if some \(x\geq s\) has \(f(x)>0\), then \(f(y)=1\) for all \(y<x\). In particular, if \(|f|>s\), we must have some such \(f(x)>0\), and so \(f(0)=1\), giving \(|\Delta f|=1\), matching that of \(h\). Otherwise, we can only have \(|\Delta f|<1\) if \(f(x)=0\) for each \(x\geq s\), i.e.\ \(f(0)+\dots+f(s-1)=|f|\). But here \(|\Delta f|\) is minimised by setting \(f(0),\dots,f(s-1)\) equal, which is exactly the ball \(h\). Thus \(h\) is optimal, as desired.

    We now define \(g:A^n\to [0,1]\) by defining each layer \(g_x(x_1,\dots,x_{n-1})=g(x_1,\dots,x_{n-1},x)\). Set \(g_1=\dots=g_{s-1}\) to be the fractional Hamming ball \(b_{k,\alpha}=b_{k,\alpha}^{(n-1,s)}\) on \(A^{n-1}\) such that \((s-1)|g_1|=|f_1|+\dots+|f_{s-1}|\), and \(g_s\) to be the fractional Hamming ball \(b_{k',\alpha'}=b_{k',\alpha'}^{(n-1,s)}\) on \(A^{n-1}\) with the same weight as \(f_s\). We take \(g_x=0\) for every \(x>s\). Note that as \(f\) is \(s\)-compressed, we necessarily have \(k+\alpha\leq k'+\alpha'\).
    
    If \(k'<k\), then certainly \(g\) is \(s\)-compressed. Note that if \(g\) was exactly a ball on \(A^n\), we would have \(k'=k-1,\alpha'=\alpha\). We divide into three cases based on whether \(g_s\) has larger weight than \(b_{k-1,\alpha}\), and whether \(g\) is \(s\)-compressed.

    \paragraph{Case 1: \(k'+\alpha'\leq k-1+\alpha\)} Note that by considering deletion in the last coordinate,
    \[|\Delta f|\geq\max_{x\in [s-1]}|f_x|\geq |b_{k,\alpha}|.\]
    Now the family \(g\) is a subfamily of \(b_{k,\alpha}\), so \(\Delta g\) is a subfamily of \(b_{k,\alpha}\). Note that the two uses of \(b_{k,\alpha}\) refer to different families according to the word lengths of \(g, \Delta g\) respectively. However, \(g_1=b_{k,\alpha}\) is a subfamily of \(\Delta g\), so indeed \(\Delta g\) is exactly \(b_{k,\alpha}\). Hence we have \(|\Delta f|\geq|\Delta g|\). But now decreasing \(k+\alpha\) decreases \(|b_{k,\alpha}|\), so the smallest deletion shadow possible in this case is obtained when \(k'+\alpha'=k-1+\alpha\), in which case \(g=h\), as desired.

    \paragraph{Case 2: \(k-1+\alpha\leq k'+\alpha'\leq k\)} Now we have \(\Delta g\neq b_{k,\alpha}\), so we show \(g\) is optimal with a different lower bound on \(|\Delta f|\). By considering deletion in coordinates other than the last one, we note that
    \[(\Delta f_x)(w)\leq(\Delta f)_x(w)\]
    for each \(x\in A\) and \(w\in A^{n-2}\). Hence
    \[|\Delta f|\geq\sum_{x=1}^{s}|\Delta f_x|.\]
    But now define the family \(f_{\text{avg}}\) on \(A^{n-1}\) by \((s-1)f_{\text{avg}}(w)=f_1(w)+\dots+f_{s-1}(w)\). Then
    \begin{align*}
        \left(\Delta f_{\text{avg}}\right)(x_1,\dots,x_{n-1}){}&={}\max_{x'\in A,i\in[n]}f_{\text{avg}}(x_1,\dots,x_{i-1},x',x_i,\dots,x_{n-1})\\
        {}&\leq{}\frac1{s-1}\sum_{r=1}^{s-1}\max_{x'\in A,i\in[n]}f_r(x_1,\dots,x_{i-1},x',x_i,\dots,x_{n-1})\\
        {}&={}\frac1{s-1}\sum_{r=1}^{s-1}\left(\Delta f_r\right)(x_1,\dots,x_{n-1}).
    \end{align*}
    In particular, we find that
    \[|\Delta f|\geq (s-1)\left|\Delta f_{\text{avg}}\right|+\left|\Delta f_{s}\right|.\]
    But by induction, these terms are respectively minimised, given the weights of \(f_{\text{avg}}\) and \(f_s\), when these are both balls, in which case \(f_{\text{avg}}=g_1\) and \(f_s=g_s\). Now as \(k-1+\alpha \leq k' + \alpha' \leq k\),  certainly \(b_{k,\alpha}\) is a subfamily of \(g\). The only further contribution to the deletion shadow given by \(g_s\) is obtained by deleting non-final coordinates. Hence the deletion shadow of \(g\) is given by
    \begin{align*}\left(\Delta g\right)(x_1,\dots,x_{n-1}){}&={}b_{k,\alpha}(x_1,\dots,x_{n-2})&\text{ if }x_{n-2}<s,\\{}&={}b_{k',\alpha'}(x_1,\dots,x_{n-2})&\text{ if }x_{n-2}=s.\end{align*}
    In particular, \(|\Delta g|\) attains the bound \((s-1)\left|\Delta g_1\right|+\left|\Delta g_{s}\right|\), so is optimal given the weights of \(g_1\) and \(g_s\). We now ask how best to choose these weights to minimise the weight of \(\Delta g\).
    Now we have \begin{align*}|g|={}&{}(s-1)\left(\left|\mathcal{B}^{(n-1,s)}(k)\right|+\alpha {{n-1}\choose {k+1}}(s-1)^{n-k-2}\right)\\{}&{}+\left(\left|\mathcal{B}^{(n-1,s)}(k-1)\right|+\alpha' {{n-1}\choose {k}}(s-1)^{n-k-1}\right),\end{align*}
    but
    \begin{align*}\left|\Delta g\right|={}&{}(s-1)\left(\left|\mathcal{B}^{(n-2,s)}(k)\right|+\alpha {{n-2}\choose {k+1}}(s-1)^{n-k-3}\right)\\{}&{}+\left(\left|\mathcal{B}^{(n-2,s)}(k-1)\right|+\alpha' {{n-2}\choose {k}}(s-1)^{n-k-2}\right).\end{align*}
    Then there are positive constants \(c_1,c_2,c_3,c_4\) such that
    \[|g|=c_1+c_2\left(\alpha\frac{n-k-1}{k+1}+\alpha'\right),\]
    \[\left|\Delta g\right|=c_3+c_4\left(\alpha\frac{n-k-2}{k+1}+\alpha'\right).\]
    In particular, given the weight of \(g\), the weight of \(g\)'s shadow is minimised when \(\alpha\) is maximised. Within this case, this is attained when \(k'+\alpha'=k-1+\alpha\), in which case \(g\) is exactly \(h\), the Hamming ball on \(A^n\).

    \paragraph{Case 3: \(k=k'\)} As above, we have
    \[|\Delta f|\geq (s-1)\left|\Delta f_{\text{avg}}\right|+\left|\Delta f_{s}\right|\geq (s-1)\left|\Delta g_1\right|+\left|\Delta g_{s}\right|.\]
    However, as above there are also positive constants \(c'_1,c'_2,c'_3,c'_4\) such that
    \[|f|=c'_1+c'_2\left(\alpha(s-1)+\alpha'\right),\]
    \[(s-1)\left|\Delta g_1\right|+\left|\Delta g_s\right|=c'_3+c'_4\left(\alpha(s-1)+\alpha'\right).\]
    But now this bound depends only on the weight of \(f\), and is constant across all of Case 3. It will therefore be sufficient to show that this exceeds the weight of \(h\) for a single value of \(\alpha\). Given the weight of \(f\), we choose \(\alpha,\alpha'\) such that \(\alpha'=0\) or \(\alpha=1\). Then by the analysis of Case 2, \(g\) attains the desired bound. However, we have already shown that this has shadow has weight at least as large as that given by \(h\), as desired.

    \,

    We now consider the case where we may have \(f_x\neq 0\) for \(x>s\). Let \(f'\) be defined by \(f'_x=f_x\) for \(x\leq s\), and \(f'_x=0\) for any \(x>s\). Notice that as \(f\) is \(s\)-compressed, \(f_x\) is a subfamily of \(f_1\) for any \(x\geq s\), since whenever \(f_x(x_1,\dots,x_{n-1})>0\) we in fact have \(f_1(x_1,\dots,x_{n-1})=1\). Thus the contribution of \(f_{s+1},f_{s+2},\dots\) to the shadow of \(f\) is only from deletion of non-final coordinates. Now we have
    \[
    \Delta f\left(x_{1},\dots,x_{n}\right)=
    \begin{cases}
    \left(\Delta f^{\prime}\right)\left(x_{1},\dots,x_{n}\right) &
    \text{if }x_{n}<m,\\ f_{x_{n}}\left(x_{1},\dots,x_{n-1}\right) &
    \text{if }x_{n}\geq m.
    \end{cases}
    \]
    In particular,
    \[\left|\Delta f\right|=\left|\Delta f'\right|+\sum_{x=s+1}^{r}\left|f_x\right|.\]
    But these terms are respectively minimised by balls. Define \(g\) by letting \(g'\) be the fractional Hamming ball \(b^{(n,s)}_{k,\alpha}\) on \(A^n\) with the same weight as \(f'\), and for each \(x>s\), letting \(g_s\) be the ball \(b^{(n-1,s)}_{k_x,\alpha_x}\) on \(A^{n-1}\) with the same weight as \(f_x\). Without loss of generality, we take \(\alpha<1\) whenever \(|f'|<s^n\), and \(\alpha_x>0\) whenever \(|f_x|\neq 0\). Note that certainly \(\left|\Delta f\right|\geq\left|\Delta g\right|.\)

    Now as \(f\) is \(s\)-compressed, we have \(k_x<k\) for each \(x>s\). But now we have
    \begin{align*}|g|={}&{}\left|\mathcal{B}^{(n,s)}(k)\right|+\alpha {n\choose{k+1}}(s-1)^{n-k-1}\\{}&{}+\sum_{x=s+1}^{r}\left(\left|\mathcal{B}^{(n-1,s)}(k_x)\right|+\alpha_x {n-1\choose{k_x+1}}(s-1)^{n-k_x-2}\right),\end{align*}
    \begin{align*}\left|\Delta g\right|={}&{}\left|\mathcal{B}^{(n-1,s)}(k)\right|+\alpha {n-1\choose{k+1}}(s-1)^{n-k-2}\\{}&{}+\sum_{x=s+1}^{r}\left(\left|\mathcal{B}^{(n-2,s)}(k_x)\right|+\alpha_x {n-2\choose{k_x+1}}(s-1)^{n-k_x-3}\right).\end{align*}
    We now treat this as an optimisation problem, given the weight of \(g\). Write
    \[\beta=\alpha {n\choose{k+1}}(s-1)^{n-k-1},\ \beta_x=\alpha_x {n-2\choose{k_x+1}}(s-1)^{n-k_x-3},\]
    so that the constraint on the weight of \(g\) becomes a constraint on \(\beta+\sum_{x=s+1}^{r}\beta_x\). There is now some constant \(C\) such that
    \[\left|\Delta g\right|=C+\frac1{s-1}\left(\frac{n-k-1}{n}\beta + \sum_{x=s+1}^{r}\frac{n-k_x-2}{n-1}\beta_x\right).\]
    But \(k_x<k\) implies that \(\frac{n-k-1}n<\frac{n-k_x-2}{n-1}\). So decreasing some \(\beta_x\) to increase \(\beta\) decreases \(\left|\Delta g\right|\). Equivalently, decreasing some \(\alpha_x\) to increase \(\alpha\) decreases \(\left|\Delta g\right|\). Thus the optimum must have either \(\alpha=1\) or all \(\alpha_x=0\). But by our choice of \(\alpha\), this implies that \(f_x=0\) for all \(x>s\).
\end{proof}

\section{Further questions}\label{sec:further}
Previously, the minimal deletion shadow was only known for families of a given size of the form \([s]^n\), for some \(s\leq r\). For each \(s<r\), Theorem \ref{thm:discreteresult} gives \(n-1\) new sizes between \(s^n\) and \((s+1)^n\) for which the minimal deletion shadow is known. A natural further question would be whether we can construct families of minimal deletion shadow for an even denser set of sizes.
\begin{question}
Do families of the form 
\[\begin{split}\big\{w\in [s]^n:\ & w\text{ contains the symbol }s\text{ at most }k\text{ times, or }\\& w\text{ contains the symbol }s\text{ exactly }k+1\text{ times and the symbol }(s-1)\text{ at most }\ell\text{ times}\big\}\end{split}\]
have minimal deletion shadow given their size?
\end{question}

Note that this does not conflict with the result of Leck. His proof rests on a counterexample at a single value, and so is very specific in what it forbids. In particular, it does not preclude dense sets of sizes at which initial segments of some ordering give minimal deletion shadow.

More generally, let \(c_x(w)\) denote the number of occurrences of the character \(x\) in the word \(w\), and let the \emph{signature} of a word be \(\sigma(w)=(c_r(w),\dots,c_1(w))\). Let \(<_L\) denote the lexicographic order, where for any two signatures \(p,q\), we say \(p<_Lq\) if \(p_i<q_i\) at the first index \(i\) where \(p_i\neq q_i\). We pose the following conjecture:
\begin{conjecture}
    For any signature \(s\), the family \(\{w\in A^n: \sigma(w)\leq_Ls\}\) has minimal deletion shadow given its size.
\end{conjecture}

The statement that families of the form \([s]^n\) have minimal deletion shadow corresponds to the special case of this conjecture where \(s\) has form \((0,\dots,0,n,0,\dots,0)\). Our Theorem \ref{thm:discreteresult} is then the special case where \(s\) has form \((0,\dots,0,k,n-k,0,\dots,0)\).

\bibliographystyle{plain}
\bibliography{deletion-shadows}

\begin{thebibliography}{1}

\bibitem{bollobas_isoperimetric_1991}
B.~Bollobás and I.~Leader.
\newblock Isoperimetric inequalities and fractional set systems.
\newblock {\em Journal of Combinatorial Theory, Series A}, 56(1):63--74, 1991.

\bibitem{bollobas_maximal_1993}
B.~Bollobás and I.~Leader.
\newblock Maximal sets of given diameter in the grid and the torus.
\newblock {\em Discrete Mathematics}, 122(1):15--35, 1993.

\bibitem{danh_structure_1996}
T~Danh and D.~Daykin.
\newblock The structure of {V}-order for integer vectors.
\newblock {\em Congressus Numerantium}, 113, 1996.

\bibitem{danh_ordering_1997}
T.~Danh and D.~Daykin.
\newblock Ordering {Integer} {Vectors} for {Coordinate} {Deletions}.
\newblock {\em Journal of the London Mathematical Society}, 55(3):417--426, 1997.

\bibitem{danh_sets_1997}
T.~Danh and D.~Daykin.
\newblock Sets of 0, 1 vectors with minimal sets of subvectors.
\newblock {\em Rostocker Mathematisches Kolloquium}, 50, 1997.

\bibitem{leck_nonexistence_2004}
U.~Leck.
\newblock Nonexistence of a {Kruskal}-{Katona} {Type} {Theorem} for {Subword} {Orders}.
\newblock {\em Combinatorica}, 24(2):305--312, 2004.

\bibitem{loomis_inequality_1949}
L.~H. Loomis and H.~Whitney.
\newblock An inequality related to the isoperimetric inequality.
\newblock {\em Bulletin of the American Mathematical Society}, 55(10):961--962, 1949.

\bibitem{raty_coordinate_2019}
E.~Räty.
\newblock Coordinate {Deletion} of {Zeroes}.
\newblock {\em The Electronic Journal of Combinatorics}, pages P3.50--P3.50, 2019.

\end{thebibliography}

\end{document}